\documentclass[12pt]{amsart}

\usepackage{amssymb,amsmath,graphics,verbatim,bm,stmaryrd}
\usepackage{latexsym}
\usepackage{eucal}
\usepackage{a4wide}

\newtheorem{theorem}{Theorem}
\newtheorem{lemma}[theorem]{Lemma}
\newtheorem{proposition}[theorem]{Proposition}

\theoremstyle{definition}

\theoremstyle{remark}

\newcommand{\bR}{\mathbb{R}}
\newcommand{\bZ}{\mathbb{Z}}

\newcommand{\End}{\mathrm{End}}

\newcommand{\II}{\mathrm{I\! I}}
\newcommand{\ind}{\mathrm{index}}

\newcommand{\pt}{\partial_t}
\newcommand{\Ric}{\mathfrak{Ric}}
\newcommand{\scal}{\mathfrak{scal}}
\newcommand{\Sch}{\mathfrak{Sch}}

\newcommand{\signature}{\mathrm{signature}}

\newcommand{\tr}{\mathrm{tr}}

\begin{document}
\title{Boundaries of locally conformally flat manifolds in dimensions $4k$}
\author{Sergiu Moroianu}
\thanks{Partially supported by the CNCS project PN-II-RU-TE-2012-3-0492.}
\address{Institutul de Matematic\u{a} al Academiei Rom\^{a}ne\\
P.O. Box 1-764\\RO-014700
Bu\-cha\-rest, Romania}
\email{moroianu@alum.mit.edu}
\date{\today}
\begin{abstract}
We give global restrictions on the possible boundaries of compact, orientable, locally conformally
flat manifolds of dimension $4k$ in terms of integrality of eta invariants.
\end{abstract}
\maketitle

\section{Introduction and statement}
There exist obstructions for $4k-1$-manifolds bounding a compact $4k$-dimensional Riemannian 
manifold with special geometry. 
\begin{itemize}
\item Chern and Simons \cite{cs} showed that the Chern-Simons 
invariant of a closed oriented $3$-manifold conformally immersed in $\bR^4$ must vanish modulo $\bZ$.
\item Atiyah, Patodi and Singer \cite{apsii} proved that the eta invariant of the odd signature operator
on a closed $3$-manifold conformally
embedded in $\bR^4$ must vanish. 
\item Long and Reid \cite{longreid1} considered a possibly nocompact hyperbolic $4$-manifold with totally
geodesic boundary. They allowed non-compact cuspidal ends, and cut out the cusp by a flat totally umbilic section.
They obtained a compact hyperbolic manifold whose boundary components are either totally geodesic or umbilic and flat.
They proved that the eta invariant of the boundary must be an even integer. 
\item Xianzhe Dai \cite{dai} showed that the eta invariant of a locally conformally flat manifold $(M,h)$ vanishes modulo $2\bZ$ 
whenever $M$ is the oriented, totally umbilic boundary of a locally conformally flat compact manifold $(X,g)$.
Dai's result contains Long and Reid's as a particular case.
\end{itemize}
With the notable exception of \cite{cs}, the essential ingredient in the proof of the
above results is the Atiyah-Patodi-Singer signature formula for manifolds with 
boundary equipped with metrics of product type near the boundary \cite{apsi}. 
In dimension $4$, this formula reads for instance:
\begin{equation}\label{signafor}
-\tfrac{1}{24\pi^2}\int_X\tr(R^2) - \tfrac{1}{2}\eta(M,h) = \signature(X)\in\bZ.
\end{equation}

Our main result here is:
\begin{theorem}\label{tve}
Let $(X,g)$ be a compact oriented locally conformally flat Riemannian manifold of dimension $4k$, 
with smooth boundary $(M,h)$. Then the eta invariant of the odd signature operator of $(M,h)$, 
and also the eta invariant of the Dirac operator if $M$ has a spin structure, belong to $2\bZ$.
\end{theorem}
Compared to the corresponding results of \cite{apsii} and \cite{dai}, we do not ask the interior metric $g$ to be flat 
(but only locally conformally flat), respectively we do not impose any restrictions on the second fundamental form 
other than the usual Codazzi-Mainardi equation. 

The proof is algebraic in nature, and relies on 
the analysis of a transgression form for the Pontriagin forms appearing 
in the Atiyah-Patodi-Singer index formula.

\section{The index formula for product metrics}

Let us recall the index formula of Atiyah, Patodi and Singer \cite{apsi}.
Let $(X,g$) be a Riemannian metric on a oriented compact manifold $X$, with boundary $M$, 
and denote by $h$ the induced metric on $M$. 
The metric $g$ is said to be of \emph{product type} if
there exists an embedding of manifolds with boundary $[0,\epsilon)\times M\hookrightarrow X$ such that the 
pull-back of the metric $g$ takes the form 
\[
g=dt^2+h.
\]
Intrinsically, this condition can be reformulated as 
\[L_\nu g=0\text{ near $M$,}\] 
where $L$ denotes Lie derivative and $\nu$ is 
the geodesic unit vector field normal to the boundary. 
\begin{theorem}[\cite{apsi}]\label{thin}
Let $(X,g)$ be a compact Riemannian manifold with boundary, 
and assume that the metric is of product type near the boundary $(M,h)$.
Let $(D,A)$ denote one of the following pairs of elliptic operators on $X$, respectively on $M$:
\begin{enumerate}
\item $D$ is the signature operator on $X$, and $A$ is the odd signature operator on $M$;
\item When $X$ has a fixed spin structure with spinor bundle $S$, 
$D$ is the Dirac operator on $(X,S)$, and $A$ is the Dirac operator on $M$ for the induced spin structure.
\end{enumerate}
Then the index of the $D$ (with the non-local APS boundary condition defined by $A$) equals
\begin{equation}\label{if}
\bZ\ni \ind(D)=\int_X \exp\left[\tfrac{1}{2}\tr\log\left(P\left(\tfrac{R}{2\pi i}\right)\right)\right]
- \tfrac{1}{2}\eta(A).
\end{equation}
where $R$ is the Riemannian curvature tensor of $g$, and $P$ is a specific Taylor series:
\[
P(x)=\begin{cases}
\frac{x}{\tanh(x)}&\text{for the signature operator;}\\
\frac{x/2}{\sinh(x/2)}& \text{for the Dirac operator.}
\end{cases}
\]
\end{theorem}
Only the monomials of degree at most $2k$ in the series $P$ contribute to the integrand.

The reader unfamiliar with the definition of the index, the signature and the Dirac operator, the eta invariant, 
the non-local APS boundary condition, and the proof of the index formula is referred to the original paper 
of Atiyah, Patodi and Singer \cite{apsi}. But in fact, none of these ingredients really matter in the present work.
Indeed, as a consequence of the index formula, the eta invariant of $A$ on $M$, 
defined in \cite{apsi} as a spectral invariant intrinsic to $M$ and the operator $A$, 
can be computed modulo
$\bZ$ using the index formula on some oriented manifold $X$ bounding $M$ (if any) with metric of product type:
\begin{equation}\label{emz}
\tfrac{1}{2}\eta(A)\equiv \int_X  e^{\frac{1}{2}\tr \log P\left(\frac{R}{2\pi i}\right)} \mod\bZ.
\end{equation}
Reversing the logical order of thought, assuming that $M$ is the boundary of some oriented $X$
we could \emph{define} the eta invariant modulo $2\bZ$ by this equation. The independence of the right-hand side 
on $(X,g)$ (with $g$ of product type) is a consequence of Hirzebruch's signature formula, 
respectively of the Atiyah-Singer index theorem.

If $g$ is locally conformally flat, the Weyl curvature tensor vanishes, and so the integrand in \eqref{emz}, 
which only depends
on the Weyl tensor, vanishes in positive degrees. It follows that for such metrics on $X$
the eta invariant of $A$ is an even integer, under the condition that $g$ is of product type. 
But we actually prove below that this last hypothesis is not necessary!

Our tool will be an algebraic extension of the index formula 
to certain metrics which are not of product type:
\begin{theorem}\label{ifW0}
The index formula \eqref{if} is valid for metrics $g$ which are not necessarily 
of product type near $M$, but are locally conformally flat to the first order near $M$, 
in the sense that the Weyl tensor of $g$ vanishes at every $p\in M$.
\end{theorem}

\section{Transgression forms}\label{sect3}

The Atiyah-Patodi-Singer index formula can be extended to metrics which are not of product type,
by adding a boundary correction term \cite{gilkey}. The rough nature of this term is well-understood: it is a polynomial
in the second fundamental form, its twisted exterior derivative, and the curvature tensor of the boundary. In dimension 
$4k=4$ the correction term is explicitly given in \cite{EGH}. This boundary term vanishes when the boundary is totally geodesic, or even just umbilic.
We want to prove that this transgression form vanishes under the hypothesis of theorem \ref{ifW0}.

Let 
\begin{align*}
[0,\epsilon)\times M\hookrightarrow X,&&(t,p)\mapsto \exp_p(t\nu)
\end{align*} 
be the embedding defined by the geodesic normal flow from the boundary, where
$\nu$ is the unit inner normal field. The pull-back of the metric $g$ takes the form 
\[g=dt^2+h(t)\]
where $h(t)$ is a smooth family of metrics on $M$ starting at $h(0)=h$.
We can then compute
\[L_{\pt} g=\pt h(t)\text{ near $M$,}\] 
Let $g_0$ be any metric on $X$ which equals $dt^2+h$ near $M$. Clearly $g_0$ is of product type, 
and shares with $g$ the same geodesics normal to the boundary.
Let $\nabla^1$, $\nabla^0$ denote the Levi-Civita covariant derivatives with respect to $g$, $g_0$, and define
\[
\theta:=\nabla^1-\nabla^0\in\Lambda^1(X,\End(TX)).
\]
Let also $\nabla^M$ denote the Levi-Civita connection of $(M,h)$. The second fundamental form $\II$ of the inclusion
$(M,h)\hookrightarrow (X,g)$ is defined for vector fields $U,V$ tangent to $M$
by
\[
\nabla^1_UV=\nabla^M_UV+\II(U,V)\pt.
\]
Notice that $\nabla^M_UV=\nabla^0_UV$ because $g^0$ is of product type. Let $W$ be the Weingarten operator,
$h(WU,V)=\II(U,V)$. Then for vectors $U,V$ tangent to $M$ we have:
\begin{align}\label{tetab}
\theta(U)V=\II(U,V)\pt,&& \theta(U)\pt = -W(U),&&\theta(\pt)=-W.
\end{align}
The last identity holds because $\pt$ is a geodesic vector field both for $g$ and $g^0$.

Consider the segment of connections
\[
\nabla^s=\nabla^0+s\theta
\]
linking $\nabla^0$ to $\nabla^1$ for $s\in[0,1]$, and 
let $R^s$ be the curvature tensor of $\nabla^s$ 
(so $R^1$ is the curvature of the initial metric $g$). We have
\begin{equation}\label{crs}
R^s=R^0+ s d^{\nabla^0}\theta +s^2 \theta^2\in \Lambda^*(X,\End(TX))[s].
\end{equation}

Here we pause to explain the notation. Whenever $\omega,\omega'\in\Lambda^*(X)$
and $B,B'\in\End(TX)$, we define the product of the endomorphism-valued forms $\omega\otimes B$ and
$\omega'\otimes B'$ as follows:
\[
\omega\otimes B\cdot \omega'\otimes B':= \omega\wedge\omega'\otimes BB'.
\]
Thus $\theta^2$ is an endomorphism-valued $2$-form. The bracket $[s]$ signifies that the dependence of
$R^s$ on $s$ is polynomial (of degree $2$). For later use, define the trace:
\[\tr(\omega\otimes B):=\tr(B)\omega\in\Lambda^*(X).
\]
We note the obvious trace identity:
\begin{equation}\label{traceid}
\tr(\omega\otimes B\cdot \omega'\otimes B')=(-1)^{\deg(\omega)\cdot\deg(\omega')}\tr(\omega'\otimes B'\cdot \omega\otimes B).
\end{equation}

\begin{lemma}\label{t2d}
Let $\{S_j\}_{1\leq j\leq 4k-1}$ be a local orthonormal basis for $TM$ consisting of eigenvectors of $W$, namely
$WS_j=\lambda_j S_j$ for $\lambda_j\in\bR$. Then on the boundary $M$ we have
\[\imath_M^*\theta^2=\sum_{i\neq j}\lambda_i\lambda_j S^i\wedge S^j\otimes\left[S^i\otimes S_j \right]=
\sum_{i< j}\lambda_i\lambda_jS^i\wedge S^j\otimes\left[S^i\otimes S_j - S^j\otimes S_i \right].\]
\end{lemma}
\begin{proof}
Evident from \eqref{tetab}.
\end{proof}
From this lemma, we compute immediately $\imath_M^*\theta^3=0$ on $M$.

The following proposition is standard:
\begin{proposition}
Let $Q$ be a polynomial. Then 
\[\tr \left(Q(R^1)\right)-\tr\left( Q(R^0)\right) = d\int_0^1 \tr\left(\theta Q'(R^s)\right)ds.\]
\end{proposition}
Here $Q(R^s)$ is an endomorphism-valued form as explained above.
\begin{proof}
Since $\frac{d}{ds}\nabla^s=\theta$, we have $\frac{d}{ds}R^s=d^{\nabla^s}\theta$. Therefore, using the trace identity,
\[\tfrac{d}{ds}\tr \left(Q(R^s)\right)=\tr\left(d^{\nabla^s}(\theta) Q'(R^s)\right).\] 
By the second Bianchi identity, 
$d^{\nabla^s}R^s=0$, so $d^{\nabla^s}Q'(R^s)=0$ and therefore 
\[d^{\nabla^s}(\theta) Q'(R^s)=d^{\nabla^s}\left(\theta Q'(R^s)\right).\]
To conclude, use the identity $\tr(d^\nabla \cdot)=d\tr(\cdot)$ and integrate from $0$ to $1$.
\end{proof}
The trace $\tr \left(Q(R^s)\right)$ is a differential form depending polynomially on $s$. 
Its exponential is well-defined, and is again a polynomial in $s$ with coefficients in $\Lambda^*(X)$. Moreover we have:
\begin{proposition}\label{difexptr}
Let $Q$ be a polynomial. Then 
\[e^{\tr \left(Q(R^1)\right)}-e^{\tr\left( Q(R^0)\right)} = d\int_0^1 \tr(\theta Q'(R^s))e^{\tr\left( Q(R^s)\right)}ds.\]
\end{proposition}
The proof goes exactly like in the preceding proposition. We use the fact that $\tr \left(Q(R^s)\right)$ is an even form and the trace identity
to obtain 
\[\tfrac{d}{ds}\exp(\tr \left(Q(R^s)\right))= d\tr(\theta Q'(R^s))\exp(\tr \left(Q(R^s)\right)).\]

\section{Transgression for locally conformally flat metrics on $X$}

The \emph{Schouten tensor} of the metric $g$ is defined by
\[\Sch=\tfrac{1}{n-2}\left(\Ric-\tfrac{\scal}{2(n-1)}\right),\]
where $n=\dim(X)=4k$. We henceforth assume that the Weyl 
component of the curvature tensor of $g$ vanishes at a point $p\in M$. At such $p$ we then have \cite[1.116]{besse}
\[R^1=\Sch\owedge g,\]
meaning that
\begin{equation}\label{KN}\begin{split}
(\Sch\owedge g)(U_1,U_2,U_3,U_4)={}&\Sch(U_1,U_4)\langle U_2,U_3\rangle + \Sch(U_2,U_3)\langle U_1,U_4\rangle\\
{}&-\Sch(U_1,U_3)\langle U_2,U_4\rangle - \Sch(U_2,U_4)\langle U_1,U_3\rangle.
\end{split}\end{equation}
In particular, we deduce the following:
\begin{lemma}\label{3v0}
Assume that the Weyl tensor of $g$ vanishes at $p\in M$. Then
for mutually orthogonal vectors $U_1,U_2,U_3,U_4\in T_p X$ we have $\langle R^1_{U_1U_2}U_3,U_4\rangle=0$.
\end{lemma}
\begin{proof}
If the vectors $U_j$'s are mutually orthogonal, all the scalar products in \eqref{KN} vanish.
\end{proof}

From $g$ we have constructed a product type metric $g_0$, and we have denoted
by $\theta$ the difference between the covariant derivatives corresponding to $g$ and to $g_0$.
Let $\imath_M:M\hookrightarrow X$ denote the inclusion map.
The main result of this section is the following vanishing result:
\begin{proposition}\label{w0p0}
If the Weyl tensor of $(X,g)$ vanishes at $p\in M$, 
then for every polynomial $Q$, the pull-back 
\[\imath_M^*\tr(\theta Q(R^s))\in\Lambda^*(M)\]
vanishes at $p$.
\end{proposition}
\begin{proof}
Let $\{S_j\}_{j=1}^{4k-1}$ be an orthonormal basis of $T_pM$ consisting of eigenvectors
for the Weingarten operator $W$. We have
\begin{align}
&\imath_M^*\theta= \sum_{i=1}^{4k-1} \lambda_i S^i\otimes [S^i\otimes\pt-dt\otimes S_i],\label{imte}\\
&\imath_M^* R^0=R^M,\nonumber\\
&\langle d^{\nabla^0}\theta(S_i,S_j)S_h,\pt\rangle=\langle d^{\nabla^0}W(S_i,S_j),S_h\rangle
=\langle R^1_{S_iS_j}S_h,\pt\rangle=:R^1_{ijh\nu}.\label{CoMa} \\
\intertext{Also from \eqref{imte}, using the fact that $\pt$ is parallel with respect to $g_0$,}
&\langle d^{\nabla^0}\theta(S_i,S_j)S_h,S_l\rangle=0.\label{dnt0}
\end{align}
Equation \eqref{CoMa} is just the Codazzi-Mainardi constraint for the isometric embedding $(M,h)\subset (X,g)$.
Since $\pt\perp M$, from  \eqref{CoMa}, \eqref{dnt0} and lemma \ref{3v0} we deduce
\begin{equation}\label{dnat}
\imath_M^*d^{\nabla^0}\theta = \sum_{a<b} S^a\wedge S^b\otimes[R^1_{abb\nu}(S^b\otimes\pt-dt\otimes S_b)
+R^1_{aba\nu}(S^a\otimes\pt-dt\otimes S_a)].
\end{equation}
\begin{lemma}\label{r04d}
Let $1\leq i,j,h,l\leq 4k-1$ be four distinct indices, and assume that the Weyl tensor of $g$ vanishes at $p\in M$.
Then $\langle R^0_{S_i S_j}S_h, S_l\rangle=0$.
\end{lemma}
\begin{proof}
From equation \eqref{crs} with $s=1$, Lemma \ref{t2d} becomes
\[
\langle R^0_{S_i S_j}S_h, S_l\rangle + \langle d^{\nabla^0}\theta(S_i, S_j)S_h, S_l\rangle
+\langle \theta^2(S_i, S_j)S_h, S_l\rangle=0.
\]
But the last two terms vanish by \eqref{dnat} and Lemma \ref{3v0}.
\end{proof}

By linearity, we may assume that $Q$ is a monomial, $Q(x)=x^l$. Notice that $\imath_M^*\theta$ and 
$\imath_M^* R^s$ are skew-adjoint endomorphism-valued forms, the latter being also of even degree. If $l$ is even,
by taking adjoints it follows that the trace of $\imath_M^* \theta \left( R^s\right)^l$ vanishes. Assume now that $l$ is odd.
We will show below that the diagonal of the form-valued endomorphism 
\[\imath_M^* \theta \left( R^s\right)^l= \imath_M^* \theta (R^0+s d^{\nabla^0}\theta+s^2\theta^2)^l\in 
\Lambda^{2l+1}(M,\mathrm{gl}(n))\]
written in the basis $\{\pt,S_1,\ldots,S_{4k-1}\}$ consists of $0$'s.
Rewrite \eqref{tetab} as
\begin{equation}\label{ect}
\imath_M^* \theta = \sum_{i=1}^{4k-1} \lambda_i S^i\otimes [S^i\otimes \pt- dt\otimes S_i],
\end{equation}
and decompose $\imath_M^*\theta$ into
\begin{align*}
\imath_M^* \theta=\theta'-\theta'',&&\theta':= \sum_{i=1}^{4k-1} \lambda_i S^i\otimes [S_i\otimes dt],&&
\theta'':=\sum_{i=1}^{4k-1} \lambda_i S^i\otimes[\pt\otimes S^i].
\end{align*}
Since $\theta''$ is the transpose of $\theta'$ and $R^s$ is skew-symmetric, it follows that 
\[
\imath_M^*\tr(\theta (R^0+s d^{\nabla^0}\theta+s^2\theta^2)^l)=2\tr(\theta '\imath_M^*(R^0+s d^{\nabla^0}\theta+s^2\theta^2)^l).
\]
Introduce the following unified notation for the $2$-forms entering in the expression of $\imath_M^*R^s$:
\begin{align*}
A^1=\imath_M^*R^0=R^M,&&A^2=\imath_M^*d^{\nabla^0}\theta,&&A^3=\imath_M^*(\theta\cdot\theta).
\end{align*}

\begin{lemma}
For every $l\geq 1$, and $\alpha_1,\ldots,\alpha_l\in\{1,2,3\}$, the $2l+1$-form 
\[
\theta' A^{\alpha(1)}\ldots A^{\alpha(l)}
\]
can be written as
\[
\sum_{i,j=1}^{4k-1} \sum_I C_{i,j,I} S^I\otimes[S_i\otimes S^j]            
\]
where $I$ is a multi-index of length $2l+1$, and the coefficient $C_{i,j,I}\in C^\infty(M)$ 
equals zero unless $i,j\in I$ and $i\neq j$.
\end{lemma}
\begin{proof}
Remark that $\theta' A^\alpha=0$ unless $\alpha=2$, because the endomorphism component 
of $\theta'$ restricted to $T_pM$ vanishes, while $A^1$ and $A^3$ map $T_pX$ into $T_pM$. 
Thus we can assume that $\alpha(1)=2$. The form
$\theta'\imath_M^*d^{\nabla^0}\theta$ can be written, using \eqref{ect} and \eqref{dnat}, as
\[
\sum_{I_1,I_2,I_3=1}^{4k-1}C_{I} S^{I} \otimes [S_{I_1}\otimes S^{I_3}]. 
\]
In the above sum, the wedge product $S^I=S^{I_1}\wedge S^{I_2}\wedge S^{I_3}$ 
clearly vanishes unless the three indices are mutually distinct, so we retain only the terms with $I_1\neq I_3$. 
Once we established this initial step, the proof proceeds by induction. 
Assume that the conclusion of the lemma holds for $l$. We want to prove it for the product with one additional 
factor $A^{\alpha(l+1)}$. We claim that under the hypotheses $i,j\in I$, $I\neq j$, the product 
\begin{equation}\label{pral} 
S^I\otimes[S_i\otimes S^j]\cdot A^{\alpha(l+1)}
\end{equation}
vanishes for $\alpha(l+1)=2$. Indeed, using \eqref{dnat}, we see that in order to have a non-zero term in the product
\[
S^I\otimes[S_i\otimes S^j]\cdot S^a\wedge S^b\otimes [R^1_{abb\nu}(S^b\otimes\pt-dt\otimes S_b)
+R^1_{aba\nu}(S^a\otimes\pt-dt\otimes S_a)],
\]
we should have either $j=a$ or $j=b$. Since by the induction hypothesis $j\in I$, 
it follows that in such a case the exterior product
$S^I\wedge S^a\wedge S^b$ vanishes. 

Using Lemma \ref{t2d}, exactly the same argument as above shows the vanishing of the product \eqref{pral} in the case 
where $\alpha(l+1)=3$, i.e., $A^{\alpha(l+1)}=A^3=\theta\cdot\theta$. 

In the remaining case $\alpha(l+1)=1$ we multiply $S^I\otimes[S_i\otimes S^j]$ to the right by $A^1=R^M$. 
By Lemma \ref{r04d},
\[R^M=\sum_{a,b,c=1}^{4k-1} S^a\wedge S^b\otimes[R_{abac}^0(S_a\otimes S^c-S_c\otimes S^a)].\]
The term corresponding to a triple $(a,b,c)$ in the product $S^I\otimes[S_i\otimes S^j]\cdot R^M$ 
is non-zero only if $j=a$ or $j=c$. 
In the first case, the wedge product $S^I\wedge S^a$ vanishes since by induction $j\in I$. Thus 
the result of the product \eqref{pral}
consists of terms of the form $S^I\wedge S^a\wedge S^b\otimes[S_i\otimes S^a]$. 
By the induction hypothesis $i\in I$. If $a=i$, again the wedge product vanishes, so we may retain only
the terms with $a\neq i$. Evidently $a\in I\cup\{a\}$, proving the induction step.
\end{proof}
This lemma implies that for $l\geq 1$, the endomorphism-valued differential form $\theta'\imath_M^*(R^s)^l$ 
is off-diagonal. For $l=0$ the same holds due to \eqref{ect}.
Its trace thus vanishes, and since $\imath_M^*\theta$ is the skew-symmetric component of $2\theta'$ and $l$ is odd,
we conclude that $\imath_M^*\tr(\theta(R^s)^l)=0$ as claimed.
\end{proof}

\section{Conclusion}

Let us derive the proofs of the statements announced in the beginning of the paper.
\begin{proof}[Proof of theorem \ref{ifW0}]
Let $g$ be a metric on $X$ which is not necessarily of product type.
Let $(D,A)$ be one of the pairs of elliptic operators from the statement of theorem \ref{thin} corresponding to
$g$ and the induced metric $h$ on $M$. Let $g_0$ be the product type metric constructed from $g$
in the beginning of section \ref{sect3} by using the normal geodesic flow from the boundary, and $D_0$
the corresponding operator on $X$.

From the index theorem for the metric $g_0$, we write
\[\ind(D_0)=\int_X \exp\left[\tfrac{1}{2}\tr\log(P(R^0/2\pi i))\right]- \tfrac{1}{2}\eta(A).\]
Define a power series $Q(x):=\tfrac{1}{2}\log(P(x/2\pi i))$ (only the truncation up to degree $4k$ matters here).
From proposition \ref{difexptr}
\[e^{\tr Q\left(R^1\right)}-e^{\tr Q\left(R^0\right)} = 
d\int_0^1 \tr\left(\theta Q'\left(R^s\right)\right)
e^{\tr Q\left(R^s\right)}\]
hence by the Stokes formula,
\[
\int_X e^{\tr Q\left(R^1\right)}-e^{\tr Q\left(R^0\right)} = \int_0^1 \int_M \tr\left(\theta Q'\left(R^s\right)\right)
e^{\tr Q\left(R^s\right)}.
\]
Assume now that the Weyl tensor of $g$ vanishes at the boundary. By proposition \ref{w0p0}, 
the above right-hand side vanishes, therefore
\[\ind(D_0)=\int_X \exp\left[\tfrac{1}{2}\tr\log(P(R^1/2\pi i))\right]- \tfrac{1}{2}\eta(A).\]
To end the proof it is enough to notice that $\ind(D_0)=\ind(D)$, by the homotopy invariance of the Fredholm index.
\end{proof}

\begin{proof}[Proof of theorem \ref{tve}]
If $g$ is locally conformally flat, its Weyl tensor vanishes on $X$, in particular it vanishes at every $p\in M$,
and so theorem \ref{ifW0} applies. Moreover, 
the Pontriagin form $\tr\log(P(R^1/2\pi i))$ vanishes identically, since in general it only depends on the Weyl tensor.
Therefore we deduce 
\[\tfrac{\eta(A)}{2}=\ind(D)\in\bZ.\]
\end{proof}

\section{Consequences}

The eta invariant of the odd signature operator is known for various classes of $3$-manifolds, including lens spaces \cite{apsii}, \cite{katase}
and hyperbolic manifolds \cite{coulson}. 
These manifolds are quotients of the standard $3$-sphere, respectively of the hyperbolic $3$-space, by some discrete group of isometries,
so they can be locally embedded isometrically in $\bR^4$. By a result of Hirsch \cite{hirsch}, every oriented compact $3$-manifold 
can be $C^\infty$ immersed in $R^4$, and is moreover the boundary of some oriented compact $X$ from Thom \cite{thom}. 
Theorem \ref{tve} implies that a closed oriented $3$-manifold
cannot be isometric to the boundary of a locally conformally flat manifold, unless the sum of the eta invariants of its connected 
components is an even integer.

For instance, the lens space $L(3,1)$ is not the boundary of a locally conformally flat $4$-manifold since its eta invariant is $-\frac{1}{3}$.
Eleven out of the $12$ closed hyperbolic $3$-manifolds from \cite[Table 2]{coulson} have non-integral eta invariant, so they do not bound
locally conformally flat $4$-manifolds.

The eta invariant of the Dirac operator on Bieberbach (i.e., flat and closed) oriented $3$-manifolds was computed by Pf\"affle \cite{pfafle}. 
It is a topological invariant of the fundamental group and the spin structure.
He found that for the manifolds $G_j$, $j\in\{2,3,4,5\}$ from the Hantsche-Wendt classification of Bieberbach groups in dimension $3$, there exist
spin structures $\sigma_j$ for which the eta invariant modulo $2$ equals $-\frac{2}{k}$, where $k\geq 2$ is the order of the holonomy group 
of $G_j$ \cite[Theorem 5.6]{pfafle}. The spin cobordism group is trivial in dimension $3$, so for each $j$ there exists a spin manifold
$(X_j,\tau_j)$ with spin boundary $(G_j,\sigma_j)$.
Theorem \ref{tve} asserts that $X_j$ cannot carry a locally conformally flat metric extending some flat metrics from $G_j$.

\end{document}